%% file: subknots.tex
\documentclass[amsmath,longbibliography,secnumarabic,
floatfix,amssymb,nofootinbib,nobibnotes,letterpaper,
11pt,tightenlines,notitlepage,showkeys,showlabels]{amsart}

\usepackage[mode=buildnew]{standalone}
\usepackage[utf8]{inputenc} \usepackage[
backend=biber]{biblatex}

\usepackage[top=1in, left=1in, right=1in, bottom=1in]{geometry}
\usepackage{showlabels}
\usepackage[svgnames]{xcolor}
\usepackage{latexsym, amsmath, amscd, amsthm}
\usepackage{mathrsfs}
\usepackage{tikz}
\usepackage{pgfplots}
\usepackage{graphicx}
\usepackage{mathtools}
\usepackage{hyperref}
\usepackage{algorithmicx}
\usepackage{algpseudocode}
\usepackage{enumerate}

\usepackage{setspace} \usepackage{caption}
\usepackage{subcaption}
\usepackage{placeins}

\usetikzlibrary{decorations.markings,backgrounds,hobby,knots,calc}

\newcommand{\knotdia}{\kappa}
\newcommand{\FlatKnotDia}{\mathscr{K}}
\newcommand{\KnotShad}{\FlatKnotDia}
\newcommand{\knotshad}{k}

\newcommand{\ArbKnotDiaUK}{\ArbClass}

\newcommand{\subknotdia}{s\kappa}
\newcommand{\SubKnotShad}{S\mathscr{K}}
\newcommand{\subknotshad}{sk}

\newcommand{\sublinkshad}{s\ell}

\newcommand{\knotgrowth}{\mu_K}

\newcommand{\ArbClass}{\mathscr{C}}

\newcommand{\arbclass}{c}

\newcommand{\Prb}{\mathbb{P}}

\newcommand{\FlagContract}[2]{#1 - #2}

\DeclareMathOperator{\Tr}{Tr}

\DeclareMathOperator{\Unk}{Unk}
\DeclareMathOperator{\argmax}{argmax}

\newtheorem{conjecture}{Conjecture}
\newtheorem*{conjecture*}{Conjecture}

\newtheorem{theorem}{Theorem}
\newtheorem*{theorem*}{Theorem}
\newtheorem{proposition}[theorem]{Proposition}
\newtheorem{lemma}[theorem]{Lemma}
\newtheorem*{lemma*}{Lemma}

\theoremstyle{definition}

\newtheorem*{definition}{Definition}

\graphicspath{{../figs/}}

\DeclareFieldFormat{url}{%
  \iffieldundef{doi}{%
    \mkbibacro{URL}\addcolon\space\url{#1}%
  }{%
  }%
}

\DeclareFieldFormat{urldate}{%
  \iffieldundef{doi}{%
    \mkbibparens{\bibstring{urlseen}\space#1}%
  }{%
  }%
}

\tikzset{->-/.style={decoration={ markings, mark=at position #1 with
      {\arrow{>}}},postaction={decorate}}}

\begin{document}
%%% HEAD MATTER
\title[]{Slipknotting in random diagrams} \author{Harrison
  Chapman} \email{hchaps@gmail.com}
\address{Department of Mathematics\\
 Colorado State University, Fort Collins CO}
% \noaffiliation{}
\date{\today}
%%%

\begin{abstract}
  The presence of slipknots in configurations of proteins and DNA has been shown to affect their functionality, or alter it entirely. Historically, polymers are modeled as polygonal chains in space. As an alternative to space curves, we provide a framework for working with subknots inside of knot diagrams via knotoid diagrams. We prove using a pattern theorem for knot diagrams that not only are almost all knot diagrams slipknotted, almost all \emph{unknot} diagrams are slipknotted. This proves in the random diagram model a conjecture yet unproven in random space curve models. We also discuss conjectures on the enumeration of knotoid diagrams.
\end{abstract}
\maketitle

%\doublespacing
%\onehalfspacing
\section{Introduction}
\label{sec:intro}

An important concern in modern geometric knot theory is that of detecting knotting in open strands. Indeed, as any open strand is topologically equivalent to the straight interval (a stronger version of this is the classical ``light-bulb'' theorem~\cite{Rolfsen1976}), knotting in an open strand is a result of its \emph{geometry}. Identifying knots in open strands is important in physics and biology: Long chain polymers, such as proteins and DNA often present as long open strands. Moreover ``knotting'' in polymers (closed or open) has been shown to affect their behavior and effectiveness: Enzymes, for example, may not function with their usual efficacy (or at all!) if they are not an appropriately knotted contortion of their protein~\cite{Rawdon2015,Rawdon2012}. Slipknotting in space curves is summarized in, e.g.,~\cite{Millett2010}.

In~\cite{Cantarella2015} we introduced the diagram model of knotting as an alternative to the usual geometric space curve models: A \emph{knot diagram} is a 4-regular topological planar map (\textit{i.e.}\ an embedded planar graph) together with sign decoration at each 4-valent vertex. The \emph{random diagram model} picks a random knot diagram with the counting measure from the set of all knot diagrams with a fixed number of vertices. Knots have been studied in the view of topological maps~\cite{Coquereaux16,ZinnJustin2009,Schaeffer2004,Jacobsen2002} and there has been work on sampling knot diagrams randomly~\cite{Dunfield2014talk,Diao2012}. We note that our model differs in that we are considering \emph{all} knot diagrams (not just equivalence classes or alternating diagrams) and that we want to be sure that we sample with the \emph{uniform} measure on diagrams of fixed size. In~\cite{Chapman2016} we show that the diagram model obeys an asymptotic pattern theorem---similar to those of self-avoiding lattice polygons and equilateral space polygons---that allowed us to show principally that knot diagrams are asymptotically composite with factors of any topological type.

We extend those definitions here to discuss subknotting and slipknots in knot diagrams as an application of this theory. Namely, as the pattern theorem for diagrams can add patterns in a completely local sense, we can show that not only are all knot diagrams almost certainly slipknotted but also that almost all \emph{unknot} diagrams are slipknotted in the random diagram model.

\section{Slipknots in diagrams}
\label{sec:slipknots}

A (knot or link) diagram can be viewed in a manner similar to that of a pdcode~\cite{Cantarella2015,Coquereaux16,Chapman2016} as a \emph{combinatorial map}: A diagram with \(n\) crossings is a set of \(4n\) flags (half-edges) together with \(2n\) edges (unordered pairs of flags which are connected together) and \(n\) crossings (a cyclically oriented quadruple of flags which meet at a point, oriented counter clockwise together with a sign in \(\{+, -\}\)), so that each flag is contained in precisely one edge and one crossing. This information determines a surface embedding of the map structure; we only consider diagrams here where we have an embedding into the oriented sphere. The collection of edges is represented as a permutation \(\tau\) on the set of flags consisting of disjoint cycle permutations of length 2, and the collection of crossings \(\sigma\) is a product of disjoint cycle permutations each of length 4. Under this view, the faces are the cycles of the permutation \(\sigma\tau\) (furthermore the permutation \(\sigma\tau\) permutes flags \emph{clockwise} about their face), and the cycles of the permutation \(\sigma^2\tau\) describe the link components of the diagram with their possible orientations; a knot diagram hence requires that \(\sigma^2\tau\) has precisely 2 cycles. A \emph{shadow} is the underlying map structure of the diagram; that is, the diagram without the crossing sign information in \(\sigma\).

The condition that the diagram embeds in the sphere hence is that \(\sigma\tau\) has precisely \(n+2\) cycles. We note that diagrams embedded into other surfaces are \emph{virtual diagrams}~\cite{Kauffman1999}. We believe that the proofs here and in~\cite{Chapman2016} should extend in some cases to different classes of virtual diagrams, but will not discuss that here. Further generalizations of knot shadows (in particular, those which are non-generic or non-self-transversal) are discussed by Turaev in~\cite{Turaev_2005}; it is likely that some such classes are amenable to the same methods.

\subsection{Knotoid diagrams}
\label{sec:opendia}

First studied by Turaev in 2010~\cite{Turaev2012}, a \emph{multi-knotoid diagram} or \emph{open diagram} is a link diagram which is missing one edge. Hence it has two ``loose'' flags (\textit{i.e.}\ flags without a parent edge) called \emph{legs}, who need not lie on the same face. 2-tangle diagrams (diagrams of tangles with two flags on a common ``exterior'' face) form a subset of multi-knotoid diagrams. We furthermore continue to require that multi-knotoid diagrams embed into the plane. For a multi-knotoid diagram, cycles of \(\sigma^2\tau\) now enumerate all of the closed components of the diagram with either orientation, while there is \emph{one} cycle corresponding to the open component traversed first forwards, then backwards. The condition that a multi-knotoid diagram is planar is now that \(\sigma\tau\) has precisely \(n+1\) cycles. A multi-knotoid diagram is a \emph{knotoid diagram} if \(\sigma^2\tau\) has precisely one cycle corresponding to the open strand component. An example of a knotoid diagram is given in Figure~\ref{fig:opendia}. The underlying map structure of a knotoid diagram, without the crossing information, is again called a \emph{knotoid shadow}.

\begin{figure}[hbtp]
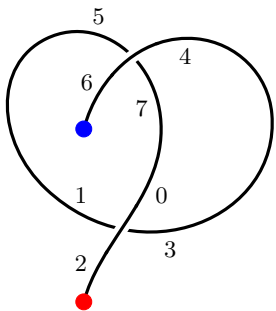

  \centering
  \includestandalone{figs/31_01_opendia}
  \caption{The knotoid diagram defined by \(\sigma = (-0123)(-4567), \tau =(07)(15)(34)\). This diagram has knot type \(\frac 123_1 + \frac 120_1\) under our \(\pm\) definition. Flag labels are drawn inside of their face, \textit{i.e.}\ their parent cycle in \(\sigma\tau = (04)(167)(235)\).}
  \label{fig:opendia}
\end{figure}

Let \(D\) be a diagram, and \(e = (ab)\) an edge in \(D\) connecting flags \(a\) and \(b\). Then \(D \setminus e\) is the 2-tangle diagram given by \(\sigma, \tau'\), where the permutation \(\tau'\) is determined by \(\tau' = \tau(ab)\). Pictorially, \(D \setminus e\) looks like \(D\) except with the edge \(e\) ``cut'' so that the flags \(a\) and \(b\) are the loose flags of the tangle. In many cases the discussion that follows may apply equally to multi-knotoid diagrams, but for the sake of brevity we will only explicitly consider knotoid diagrams.

\begin{figure}[hbtp]
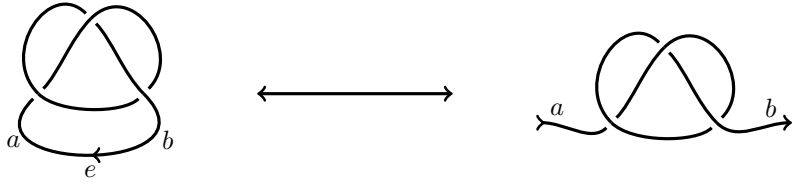

  \centering
  \includestandalone{figs/open_cut}
  \caption{A diagram \(D\) and its open version cut at \(e\), \(D \setminus e\). Either diagram has \(3_1\) as its respective knot type.}
  \label{fig:opencutdia}
\end{figure}

\begin{definition}
  Let \(S\) be a knotoid diagram of at least one crossing and \(a, z\) its loose flags. Say that \(a\) is connected to a crossing \(v\) joining flags \(a,b_1,c_1,d_1\), oriented counterclockwise (namely, flag \(c_1\) meets the crossing opposite \(a\)). Say that the latter three flags are part of edges \(e_b = (b_1b_2)\), \(e_c = (c_1c_2)\), and \(e_d = (d_1d_2)\).

  Then the \emph{contraction \(\FlagContract{S}{a}\) of \(S\) by flag \(a\)} is the new knotoid diagram produced by deleting crossing \(v\), edges \(e_b, e_c, e_d\), and flags \(a, b_1, c_1, d_1\) and inserting the edge \(e = (b_2d_2)\) joining flags \(b_2\) and \(d_2\). \(\FlagContract{S}{a}\) is a new knotoid diagram with one fewer crossing, two fewer edges, and loose flags \(c_2\) and \(z\).

  \begin{figure}[hbtp]
    \centering
    \begin{subfigure}{0.45\linewidth}
      \centering
      \includestandalone[width=0.666\textwidth]{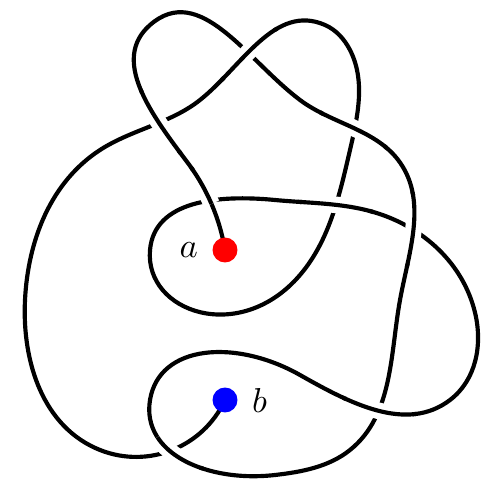}
      \label{fig:k10_146_open}
    \end{subfigure}
    \hfill
    \begin{subfigure}{0.45\linewidth}
      \centering
      \includestandalone[width=0.666\textwidth]{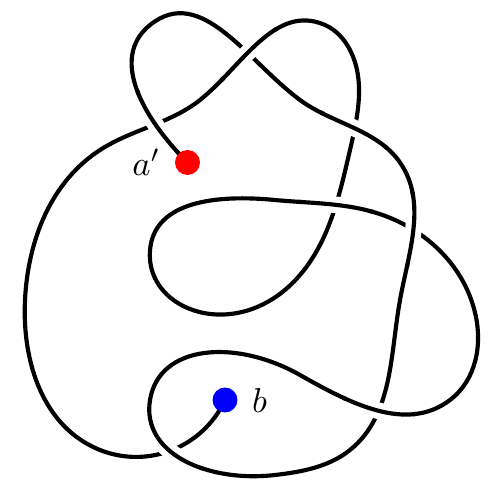}
      \label{fig:k10_146_open_contracta}
    \end{subfigure}
    \caption{a knotoid diagram \(D\) and its contraction along flag \(a\), \(D-a\).
    The diagram \(D\) has knot type \([D] = \frac 124_1 + \frac 120_1\),
    while the knot type of the contraction is \([D-a] = 0_1\).}
    \label{fig:opendiacontract}
  \end{figure}

  If \(S\) has precisely one crossing, then its contraction (by either flag) is the ``trivial'' knotoid diagram.

  If \(T\) is a knotoid diagram with no more crossings than \(S\), then say \(S\) \emph{contains \(T\) as a knotoid diagram}, \(T \le S\), if \(S = T\) or there exists a series of contraction operations on \(S\) that produces the knotoid diagram \(T\). If \(D\) is a diagram with as many or more crossings than \(T\), then say \(D\) \emph{contains T} (\(T \le D\)) as a knotoid diagram if there exists an edge \(e\) in \(D\) so that \(D\setminus e\) contains \(T\).

  \begin{figure}[hbtp]
    \centering
    \begin{subfigure}{0.45\linewidth}
      \centering
      \includestandalone[width=0.666\textwidth]{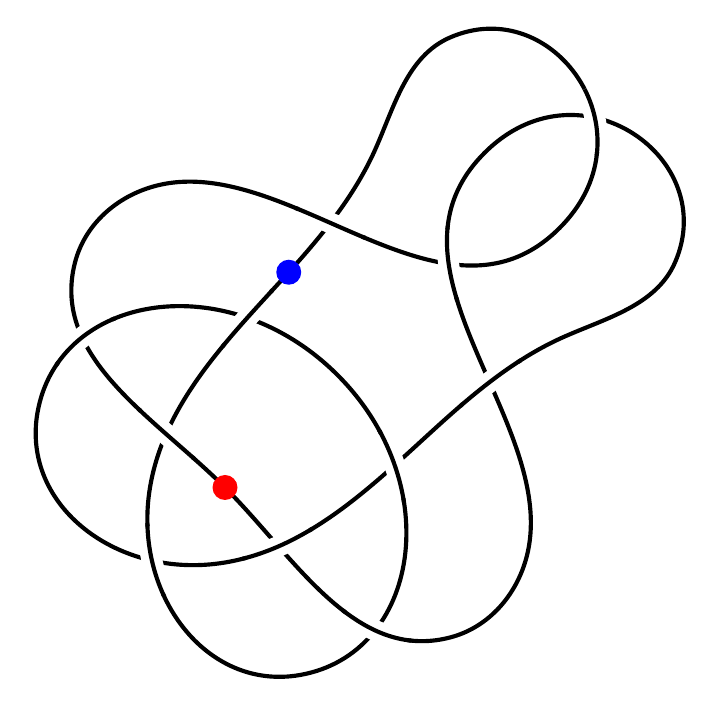}
      \caption{The knot diagram \(D\).}
    \end{subfigure}
    \hfill
    \begin{subfigure}{0.45\linewidth}
      \centering
      \includestandalone[width=0.666\textwidth]{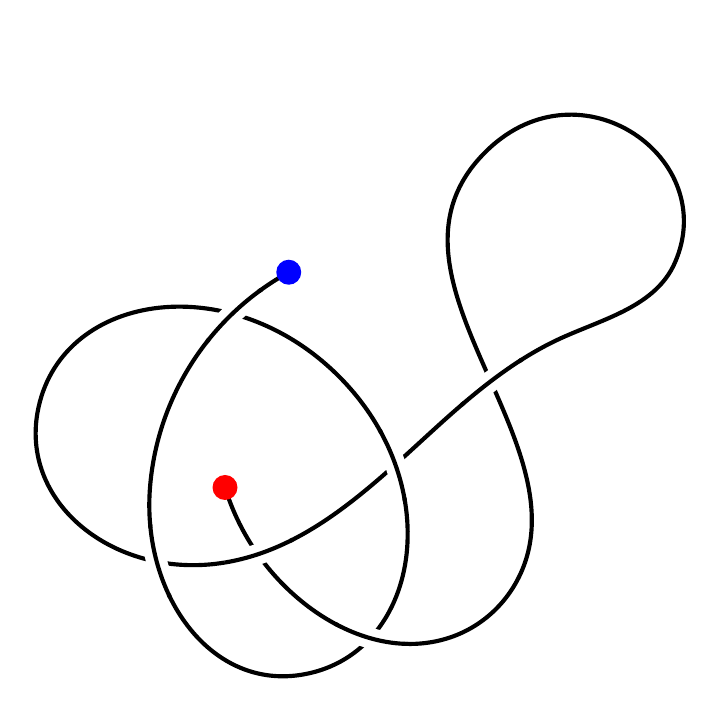}
      \caption{A subdiagram \(S\) of \(D\).}
      \label{fig:k11a135_open}
    \end{subfigure}
    \caption{A minimal knot diagram \(D\) for 11a135 contains \(S\) as a
      subdiagram. As seen in Figure~\ref{fig:diaoverunderclose}, the knot type
      of \(S\) is \([S] = \frac 124_1 + \frac 120_1\).}
    \label{fig:opensubdiacontain}
  \end{figure}
\end{definition}

This contraction operation is equivalent to the forbidden moves \(\Omega_\pm\) (also called \(\Phi_\pm\) in~\cite{Gugumcu_2017}). This is just one example of a contraction operation for diagrams. Provided a different contraction operation behaves similarly enough, one can obtain the same results on subknotting presented here.

As combinatorial objects, one hopes to find an enumeration of knotoid shadows and diagrams. Counting shadows is sufficient: Diagrams of \(n\) crossings are a \(2^n\)-fold covering of shadows of \(n\) crossings (only differing by crossing choice). In the case of link shadows and multi-knotoid shadows, enumeration is complete and simple. The number of link shadows with \(n\) crossings is known from a bijection with quadrangulations~\cite{Tutte1963,Brezin1978} to be,
\[\ell_n = \frac{2(3^n)}{(n+2)(n+1)}\binom{2n}{n} \underset{n\to\infty}{\sim}\frac 2{\sqrt \pi}12^n n^{-5/2}.\]
In fact, link shadows are \((n+2)\)-to-2 covered by a class of blossom trees~\cite{Schaeffer1997,SchaefferPlanarMap}. This class of blossom trees counts multi-knotoid shadows~\cite{Bouttier_2002,Bouttier2003}, providing the counts,
\[\sublinkshad_n = \frac{3^n}{n+1}\binom{2n}n \underset{n\to\infty}{\sim}\frac 1{\sqrt \pi}12^n n^{-3/2}. \]
Notably, the asymptotic exponential growth rates in the leading-order terms are equivalent, while the power law growth terms differ by a factor of \(n\). Enumeration of knot shadows is more complicated; a conjecture of Schaeffer and Zinn-Justin~\cite{Schaeffer2004} proposes that the number of knot shadows grows asymptotically as,
\[\knotshad_n \underset{n\to\infty}{\sim} C \knotgrowth^n n^{\gamma-2}, \]
where \(\gamma = \frac{1-\sqrt{13}}{6}\), and \(\knotgrowth \approx 11.416\pm 5\) by numerical estimates. Existence of \(\knotgrowth = \lim_{n\to\infty}{k_n^{1/n}}\) is known~\cite{Chapman2016} and numerical data continues to back up the conjecture, but there is no rigid proof of the asymptotic formula. Less is known about the number of knotoid shadows besides that \(\knotshad_n \le \subknotshad_n\), although it is reasonable to guess an asymptotic formula for leading terms,
\[\subknotshad_n \underset{n\to\infty}{\sim} C_I \mu_I^n n^{\alpha-2}. \]
In parallel with the known results for link shadows and multi-knotoid shadows, one could guess \(\alpha = \gamma+1\), although we have little data on this. On the other hand, we know:
\begin{theorem}
  The limit \(\lim_{n\to\infty}{\subknotshad_n^{1/n}} = \mu_I\) exists.
  \label{thm:knotoidgrowth}
\end{theorem}
\begin{proof}
  This is a proof by Fekete's lemma on supermultiplicative sequences. First, as \(\subknotshad_n \le \sublinkshad_n\), \(\limsup_{n\to\infty}{\subknotshad_n^{1/n}}\) is bounded. Let \(n,m \in \mathbb N\). We will show that \(\subknotshad_{n+m} \ge \subknotshad_n\subknotshad_m\) by providing a reversible mapping \(\SubKnotShad_n \times \SubKnotShad_m \hookrightarrow \SubKnotShad_{n+m}\).

  Let \(K_n, K_m\) be knotoid diagrams of \(n\) and \(m\) crossings, respectively. Let the knotoid diagram \(K\) be formed by joining the head of \(K_n\) to the tail of \(K_m\). The unique open link component of \(K\), together with the tail leg of \(K\) provides a way to index all flags of \(K\) starting from \(0\) in link-component ordering from tail to head. As \(K\) was produced through a composition of two diagrams of \(n\) and \(m\) crossings respectively, the edge \(((4n-1)\;(4n))\) is necessarily a disconnecting edge of \(K\), and removing this edge yields back the original ordered pair of diagrams \(K_n\) and \(K_m\). Together with Fekete's lemma, this yields the result.
\end{proof}
Additionally, we will present some numerical data in Section~\ref{sec:asymptotics} on the conjecture,
\begin{conjecture*}
  The connective constants for knot shadows and knotoid shadows are the same. That is,
  \[\lim_{n\to\infty}{\subknotshad_n^{1/n}} = \mu_I = \knotgrowth.\]
\end{conjecture*}
This conjecture is consistent both with asymptotic growth rates of multi-knotoid diagrams as well as with what is known in the case of self-avoiding walks as compared to self-avoiding polygons~\cite{Hammersley1961}.

\subsection{Subknots and slipknots in diagrams}
\label{sec:subknots}

For open knots in space, knot type is defined as a probability distribution of knot types of actual knots obtained by closing the open knot in all ways dictated by some choice of closure scheme.

An advantage of working with knotoid diagrams as opposed to open knots in space is that there exists a designated normal direction to the thickened plane in which the diagram resides. This suggests the following reasonable definition of the knot type of a knotoid diagram:
\begin{definition}
  Consider a knotoid diagram \(S\) with loose flags \(a, z\). An \emph{open knot type} of \(S\) is an equivalence class of \(S\) under some equivalence relation. We only consider open knot types who satisfy analogous characteristics to the MDS method~\cite{Millett2005, Millett2010} of identifying subknots in open space polygons:
  \begin{enumerate}
  \item The open knot type of a knotoid diagram does not change under small perturbations.
    % Open knot diagrams are discrete and hence there are no perturbations; this property is hence satisfied vacuously.
  \item The open knot type of a knotoid diagram which is ``almost closed'' is the same as that of the ``obvious closure''. That is, the type of the knot diagram \([D]\) is equivalent to the type of the knotoid diagram \([D \setminus e]\).
    % As the knot type of \(D\setminus e\) is precisely the same as the knot type of \(D\), for any edge \(e\) in \(D\), this property is satisfied.
  \end{enumerate}
\end{definition}

Much of what follows holds independent of choice of equivalence relation. 

\begin{definition} We provide some standard definitions of open knot type:
  \begin{enumerate}
  \item The following were first considered by Turaev: There exists a knot type \([S_{+}]\) called the \emph{overpass closure knot type} of \(S\) represented by any of the diagrams produced by adding a generic non-self intersecting curve connecting \(a\) to \(z\) and passing \emph{over} every edge of \(S\); the knot type \([S_+]\) does not depend on the choice of closing curve by Reidemeister moves (indeed, if the knotoid diagram represents an interval who lives entirely in a thickened sphere inside \(S^3\), then these closing flags can be made to exist entirely outside the sphere). Similarly, there is a knot type \([S_{-}]\) called the \emph{underpass closure knot type} of \(S\) represented by diagrams similarly closed with a curve connecting \(a\) to \(z\) but passing \emph{under} every edge in \(S\).
  
  \item In line with the probabilistic knot types used to study open knots in space~\cite{Millett2005,Rawdon2015} we make the following definition: The \emph{over-underpass knot type of a knotoid diagram \(S\)} is the normalized probability sum of knot types \(\frac 12 [S_{+}] + \frac 12 [S_{-}]\) called the \emph{knotting spectrum}.

  \item
    We can also define more general open knot types based on the underlying cell structure of the open diagram \(S\). The \emph{min-closure knot type} of a knotoid diagram is the normalization of the sum,
    \[\sum_{\rho} \sum_{\sigma} [S_{\rho,\sigma}],\]
    where the first summation is over minimal non-self-intersecting endpoint closures \(\rho\) of the knotoid diagram which are minimal in that they cross the fewest edges of \(D\) and the second summation is over all possible sign assignments for new crossings produced by closing \(D\) with \(\rho\).

  \item Finally, we have the \emph{knotoid type} of the knotoid diagram \(S\), which is its equivalence class under the relations generated only by the standard Reidemeister moves acting locally on portions of the diagram away from its legs, and disallowing the forbidden moves of contraction. Knotoid type is \emph{different} depending on whether knotoid diagrams are drawn in the plane or the sphere. As our diagrams are all assumed to be spherical, we will only consider spherical knotoid type and suppress the descriptor in the following discussion. Like knot types, knotoid types have invariants by results of Turaev~\cite{Turaev2012} and G\"ug\"umc\"u and Kauffman~\cite{Gugumcu_2017}, (namely, computable polynomial invariants) which aid in knotoid classification. Furthermore, although they are a relatively new topic of study, there is work on applying knotoids to study entanglement of proteins and other polymers~\cite{G_g_mc__2017,Goundaroulis_2017,Goundaroulis_2017b}.
  \end{enumerate}

  In all cases, an open knot type is a vector in the positive orthant of the real vector space generated by some basis of types. It hence comes with an inner product \(\langle \cdot, \cdot \rangle\) defined as the usual dot product. Our consideration that open knot types be interpreted as probabilities insists furthermore that an open knot type is a unit vector.
  
  \begin{figure}[hbtp]
    \centering
    \begin{subfigure}{0.45\linewidth}
      \centering
      \includestandalone[width=0.666\textwidth]{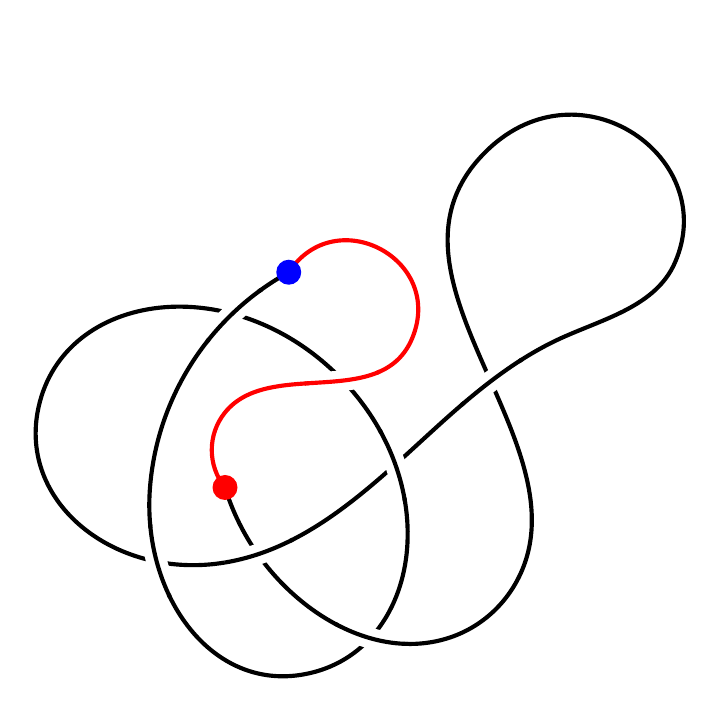}
      \caption{An over closure of \(S\), with knot type \(4_1\)}
      \label{fig:k11a135_open_trivoverclose}
    \end{subfigure}
    \hfill
    \begin{subfigure}{0.45\linewidth}
      \centering
      \includestandalone[width=0.666\textwidth]{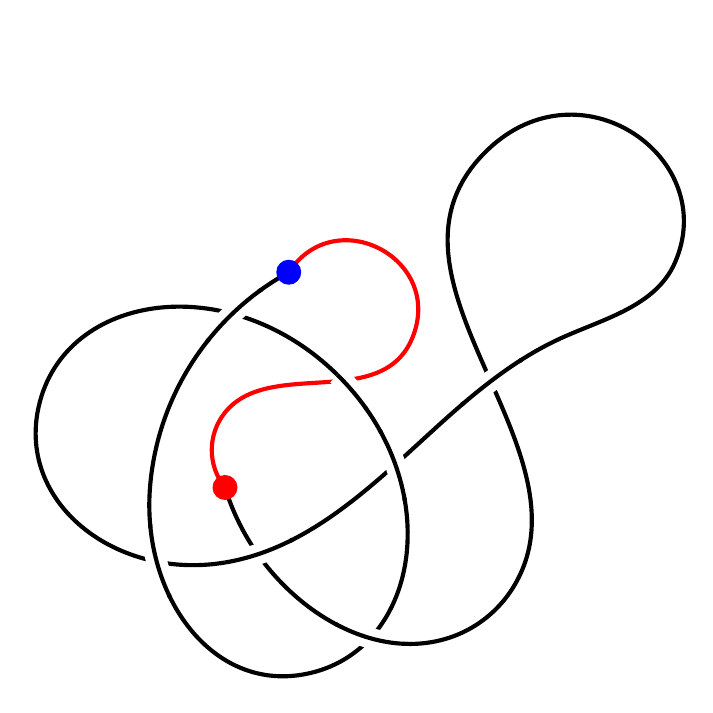}
      \caption{An under closure of \(S\), with knot type \(0_1\)}
      \label{fig:k11a135_open_trivunderclose}
    \end{subfigure}
    \caption{Closures of the knotoid diagram \(S\) from
      Figure~\ref{fig:k11a135_open}: \(S_{+}\) and \(S_{-}\).}
    \label{fig:diaoverunderclose}
  \end{figure}

  For any open knot type in the discussion that follows, we will say that \([S]\) is (strongly) trivial \([S] = 0_1\), weakly trivial if \(\langle [S], 0_1 \rangle > 0\) and probabilistically trivial if \(\langle [S], 0_1 \rangle > 1/2\). 
\end{definition}

Let \([\cdot]\) denote one of the above open knot types.

\begin{definition}
  A diagram \(D\) \emph{(strongly) contains \([K]\) as a subknot} if there exists a knotoid diagram \(S\) contained in \(D\) whose knot type is precisely \([S] = [K]\). For certain open knot types, we say the diagram \(D\) \emph{weakly contains} a knot type \([K]\) if the inner product \(\langle [D], [K]\rangle\) is nonzero, or \emph{probabilistically contains} a knot type \([K]\) if \(\langle [D], [K]\rangle \ge 1/2\).
\end{definition}

There is a chain of implications; strong implies probabilistic implies weak, although converses are not necessarily true. We will talk about diagrams containing knot types without specification of strength; the results in this paper do not depend on one or the other.

Millett~\cite{Millett_2016,Millett_2017} has a definition of knot type for subdiagrams of knot diagrams which is equivalent to our definition of the knot type of the over closure. Indeed, this is as any \emph{ascending} (possibly self-intersecting) over-closure path will produce a closed diagram with the same knot type, \textit{c.f.}\ Figure~\ref{fig:k11a135_open_millettoverclose}. Millett asks how subknot populations depend on the choice of open knot type; we will see that asymptotically and probabilistically at least, the choice does not matter.

\begin{figure}
  \centering
    \includestandalone[width=0.3\textwidth]{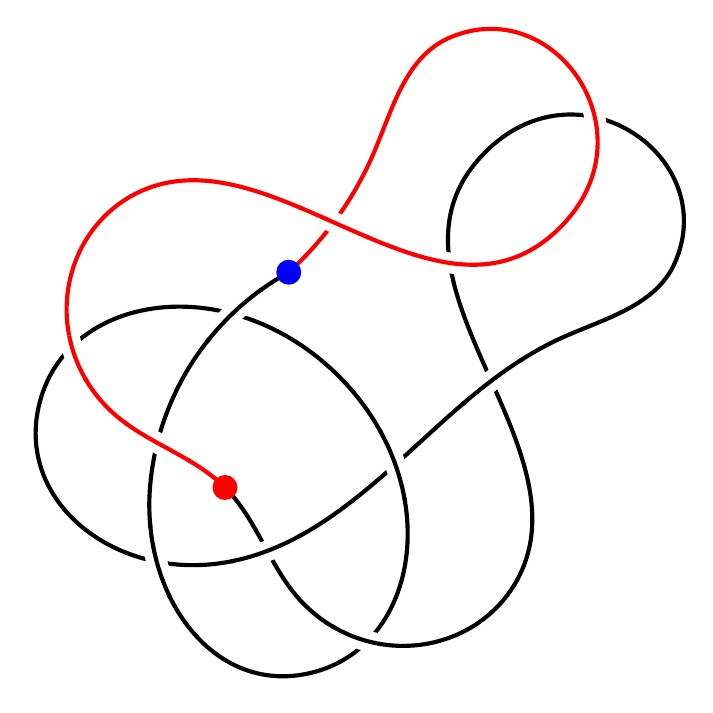}
  \caption{The Millett ascending closure of the subdiagram \(S\) inside of
    \(D\) in Figure~\ref{fig:opensubdiacontain}.}
  \label{fig:k11a135_open_millettoverclose}
\end{figure}

We note that the over-underpass definition of knot type behaves particularly nicely under contraction; if \(a\) is a loose flag of \(S\), then \(\langle [S], [\FlagContract{S}{a}] \rangle \ge \frac 12\). That is, if \(S\) is a knotoid diagram of type \( \frac 12[K_1] + \frac 12[K_2] \),there exists a knot type \([X]\) (it may be \([K_1], [K_2]\), or neither) so that \([\FlagContract{S}{a}]\) is either \(\frac 12[K_1] + \frac 12[X]\) or \(\frac 12[X] + \frac 12[K_2]\). Additionally, for the min-closure definition, we have that \(\langle [S], [\FlagContract{S}{a}] \rangle > 0\).

Different choices of open knot type behave differently with the usual Reidemeister moves:
\begin{theorem}
  In the case where open knot type is one of the \emph{overcrossing closure}, \emph{undercrossing closure}, or the \emph{over-undercrossing closure}, open knot type is invariant under the usual Reidemeister moves (\textit{i.e.}\ those which do not interact with the loose ends).

  In the case where open knot type is the \emph{min-closure} knot type, the min-closure knot type is not invariant under Reidemeister moves. However, we have that any two knotoid diagrams \(S, S'\) related by a sequence of Reidemeister moves are not orthogonal in the vector space of knot types; \(\langle [S], [S'] \rangle > 0\).

  By definition, any two knotoid diagrams with the same \emph{knotoid type} are furthermore related by a sequence of Reidemeister moves.
\end{theorem}
\begin{proof}
  Let \(D\) and \(D'\) be two knotoid diagrams which are related by a single Reidemeister move, \(\phi\). As the knot type \([D_+]\) is independent of choice of over-closure, we may choose \(D_+\) to be an over-closure which does not intersect the region to which the Reidemeister move \(\phi\) applies on \(D\). Hence \(\phi\) lifts to a Reidemeister move on \(D_+\), producing \(\phi(D_+) = D'_+\), a closure of \(D'\). Hence \([D_+] = [D'_+]\) as (usual) knot type is Reidemeister invariant. This extends to both under-closures and arbitrary sequences of Reidemeister moves, proving the claim.

  The result about min-closure knot types follows from that the min-closure knot type of a diagram \([D]\) obeys \(\langle [D], [D_+] \rangle > 0\). So given \([D']\) which differs by a sequence of Reidemeister moves we have \(\langle [D], [D'] \rangle = \langle [D_+], [D'_+] \rangle + K\), where \(K > 0\). The knotoid diagram in Figure~\ref{fig:minclosure_noninv} shows that the min-closure knot type is not invariant under the Reidemeister moves: After removing bigon \(B\) with a Reidemeister II move, the red min-closure presents as a prime alternating knot with minimal crossing number 5. On the other hand, reducing the knotoid diagram by removing bigon \(B\) yields a knotoid diagram with unique min-closure knot-type exactly \(3_1\).
  \begin{figure}
    \centering
    \includestandalone[width=0.3\textwidth]{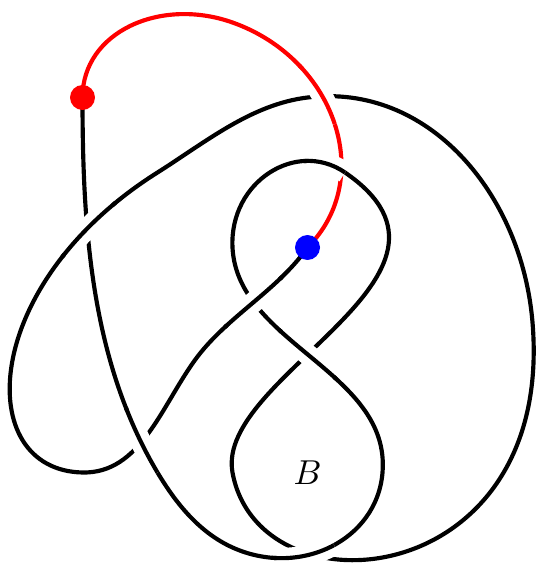}
    \caption{A knotoid diagram (with a min-closure strand) showing that min-closure is not Reidemeister invariant.}
    \label{fig:minclosure_noninv}
  \end{figure}
\end{proof}

Indeed, this are only some \emph{choices} of definition of the open knot type for knotoid diagrams. Just as in the case of space polygons, one can construct other definitions that satisfy the same conditions as the MDS method. The results that follow \emph{do not depend} on the definition chosen.

In a loose physical sense, a slipknot ``appears'' in a piece of string if, as the string is pulled taut, the slipknot unties itself. We define this concept for diagrams.
\begin{definition}
  Let \([K]\) be a knot type and \(D\) a diagram (open or closed). \(D\) \emph{contains \([K]\) as a slipknot} if there is a subdiagram \(U \le D\) where \(U\) contains \([K]\) as a subknot but \(\langle [U], 0_1 \rangle = p > 0\). The slipknot is contained \emph{weakly} if \(p > 0\), \emph{probabilistically} if \(p \ge 1/2\), or \emph{strongly} if \(p = 1\). The knotoid diagram \(S \le U \le D\) which represents the knot type \([K]\) is an \emph{ephemeral knot}.
\end{definition}
This definition is the analogue of that in~\cite{Millett2010} in the case of diagrams.

\section{Results}
\label{sec:results}

\subsection{In General}
\label{sec:general}

For the cases of classes of arbitrary link diagrams and knot diagrams (whether prime, reduced, or otherwise), we have shown in~\cite{Chapman2016} that there exist pattern theorems similar to those of Kesten for self avoiding lattice walks.

For example, in the case of the class of all knot diagrams; (\(|D|\) is the number of crossings, or \emph{size} of a diagram \(D\))
\begin{theorem*}[Pattern Theorem for general knot diagrams~\cite{Chapman2016}]
  Let \(T\) be a 3-edge-connected 2-tangle diagram of precisely one link component. Then as \(T\) can be connect-summed to any edge of a knot diagram \(D\) to produce a knot diagram (\textit{i.e.}\ knot diagrams are closed under connect summation with \(T\)), there exists \(c > 0\) and \(1 > d > 0\) so that
  \[ \Prb(D \text{ contains \(\le c|D|\) copies of T as a tangle}) <
    d^{|D|}. \]
  Namely, asymptotically almost surely, a diagram \(D\) contains \(T\) as a subtangle (in fact, a linear proportion).
\end{theorem*}

We discuss similar hypotheses and ``insertion'' constructions that yield Pattern Theorems for different classes of diagram discussed in~\cite{Chapman2016}. We first show that we can construct tangles that insert slipknots into diagrams while preserving closure within these classes;

\begin{lemma}
  \label{lem:doublinglemma}
  Let \([K]\) be a fixed knot type. Then there exists a reduced 2-tangle diagram \(D_2\) whose insertion into a knot diagram by edge replacement produces a new diagram which contains \([K]\) as a slipknot. Additionally, there exists a prime 4-tangle diagram \(D_4\) whose insertion into a prime knot diagram by crossing replacement produces a new diagram which contains \([K]\) as a slipknot. If additionally \([K]\) is an alternating knot type, \(D_4\) is an alternating 4-tangle which can be inserted appropriately into alternating diagrams.
\end{lemma}
\begin{proof}
  Given the knot type \([K]\), let \(D\) be a prime diagram representing \([K]\). Then for an arbitrary edge \(e\), \(D \setminus e\) is a knotoid diagram of open type \([K]\).

  We produce the reduced 2-tangle diagram \(D_2\) by replacing each crossing of \(D\) with a doubled ribbon as in Figure~\ref{fig:dia_doubling}, and joining one adjacent pair of loose legs. Following either of the doubled strands halfway yields the subknot for \([K]\); proceeding all the way along both doubled strands yields the slipknot.

  \begin{figure}
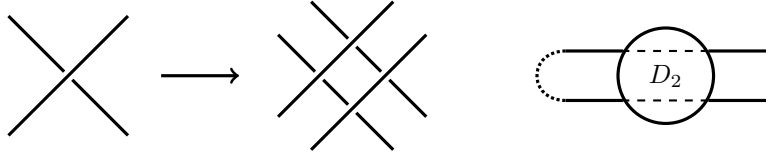

    \centering
    \includestandalone{figs/tang_doubling}
    \caption{A scheme for doubling a knot diagram of \([K]\) to introduce a slipknotted 2-tangle. A pair of adjacent external legs after the doubling process is joined by the dotted line.}
    \label{fig:dia_doubling}
  \end{figure}

  We produce the prime 4-tangle diagram \(D_4\) by modifying \(D_2\) by adding a strand passing over the diagram as in Figure~\ref{fig:dia_tripling} that keeps the tangle prime and whose insertion corresponds at the knot type level to a connected sum of the \(D_2\) portion.
  \begin{figure}
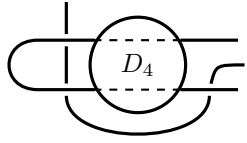

    \centering
    \includestandalone{figs/tang_double4c}
    \caption{A scheme for modifying a doubled knot diagram of \([K]\) to introduce a slipknotted 4-tangle.}
    \label{fig:dia_tripling}
  \end{figure}
  
  In the case that \(D\) is alternating, we produce the prime alternating 4-tangle diagram \(D_4^a\) by replacing each crossing of \(D\) triple weaved ribbon as in Figure~\ref{fig:dia_alttripling}, and joining one adjacent pair of outer edges (producing a new crossing with the inner edge). As the original diagram \(D\) was alternating, \(D_4\) can be seen to be alternating as well. Following either of the outer strands halfway yields the subknot for \([K]\); proceeding all the way along both outer strands yields the slipknot.
  \begin{figure}
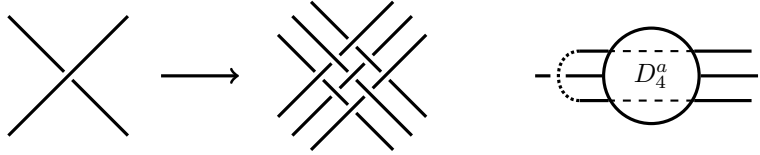

    \centering
    \includestandalone{figs/tang_tripling}
    \caption{A scheme for tripling a knot diagram of \([K]\) to introduce a slipknotted alternating 4-tangle. The external crossing is to be added to agree with the alternating crossing scheme.}
    \label{fig:dia_alttripling}
  \end{figure}
\end{proof}

With this then we can see,

\begin{theorem}
  \label{thm:arbslipknots}
  Let \([K]\) be any knot type, and \(\ArbClass\) one of the above classes of diagram (prime, reduced, or general; alternating or not; open or closed). Then there exists \(c > 0\) and \(1 > d > 0\) so that
  \[ \Prb(D \text{ contains \(\le c|D|\) copies of \([K]\) as a subknot}) < d^{|D|}, \]
  so that the probability that a random diagram in \(\ArbClass\) contains the knot type \([K]\) as a subknot and slipknot goes to one as the complexity of the diagram (\textit{i.e.}\ the number of vertices) goes to infinity.
\end{theorem}
\begin{proof}
  The result is a direct consequence of the pattern theorem, together with the construction for the appropriate tangle of Lemma~\ref{lem:doublinglemma}. We provide an example (Figure~\ref{fig:trefslipknot}) for ``edge replacement'' (connect summation) which shows that trefoil slipknots are common in general and reduced diagrams.
  \begin{figure}[hbtp]
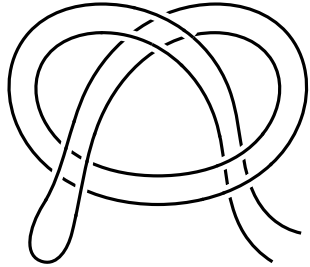

    \centering
    \includestandalone{figs/3_1_double}
    \caption{A tangle \(T\) which shows that trefoil slipknots are
      common in certain diagram classes.}
    \label{fig:trefslipknot}
  \end{figure}
\end{proof}

In particular, as every prime (reduced) alternating diagram is minimal, this answers in the affirmative the first and second conjectures of Millett and Jablan~\cite{Millett_2016,Millett_2017} in the case of minimal prime \emph{alternating} knot diagrams.

\subsection{In Unknots}
\label{sec:unknots}

As we are considering \emph{slipknotting} rather than knotting proper, we are able instead to consider classes of diagrams with \emph{fixed knot type} \([K]\)! Consider the class of knot diagrams representing the unknot. Indeed, any tangle \(T\) whose inclusion in a diagram \(D\) would be sufficient to show that \(D\) would be knotted would be immediately excluded from a pattern theorem for unknot diagrams! Tangles which insert slipknots \emph{but do not introduce any new knot type factors} however, namely those constructed in Lemma~\ref{lem:doublinglemma}, may indeed be admissible. Hence, we can show,
\begin{theorem}
  \label{thm:unkslipknots}
  Let \([K]\) be any knot type, and \(\ArbClass\) the class of either general, prime, or reduced unknot diagrams. Then there exists \(c > 0\) and \(1 > d > 0\) so that
  \[ \Prb(D \text{ contains \(\le c|D|\) copies of \([K]\) as a
      subknot}) < d^{|D|}, \]
  so that the probability that a random diagram in \(\ArbClass\) contains the knot type \([K]\) as a subknot goes to one as the complexity of the diagram (\textit{i.e.}\ the number of vertices) goes to infinity.
\end{theorem}
\begin{proof}
  Once we show the pattern theorem for these classes of unknot diagrams, we would have the result by the proof of Theorem~\ref{thm:arbslipknots}.
\end{proof}
Notice that this theorem answers Conjectures 2.12 and 2.19 from~\cite{Millett2010} in the affirmative for random unknot diagrams, and highlights an important property of the diagram model; attachments of patterns can be \emph{strongly} local (e.g.\ connect summation of a 2-tangle to an edge doesn't affect any other edges in a diagram). The expected attachment operations mentioned in the proof of Theorem~\ref{thm:arbslipknots} above can be seen easily to be viable for unknots (provided tangles whose insertion does not change the knot type of a diagram). Hence all that is required to prove a pattern theorem for unknot diagrams (cf.~\cite{Chapman2016}) is;
\begin{proposition}
  Let \(\ArbClass\) be the class of either general, prime, or reduced
  unknot diagrams, and let \(\arbclass_n\) be the count of
  \(n\)-crossing diagrams in \(\ArbClass\). Then the limit
  \[
    \lim_{n\to\infty}{\arbclass_n^{1/n}}
  \]
  exists and is equal to \(\lim_{n\to\infty}{\arbclass_n^{1/n}} < \infty\).
\end{proposition}
\begin{proof}
  The proof is similar to those in~\cite{Chapman2016}. Either of the composition constructions can be redefined for diagrams so that they only produce diagrams whose knot type \emph{is the connect sum of its components' knot types}; hence unknot diagrams only produce unknots. This provides the super-multiplicativity hypothesis that \(\arbclass_n\arbclass_m \le \arbclass_{n+f(m)}\), where \(f(m)\) is either \(m\) or \(m+2\), and proves the claim.
\end{proof}

\subsection{In Other Fixed Knot Types}
\label{sec:fixedtype}

In fact, it is this observation which shows why the proof strategy \emph{does not} immediately apply in general to classes \(\ArbClass\) of diagrams of some \emph{other}, nontrivial knot type \([K]\). Indeed, to apply the theorems in~\cite{Chapman2016} one would have to provide a composition that fixes knot type. This trouble mirrors other models of random knots, namely self avoiding polygons: The ``growth rates'' of all self avoiding polygons and those which represent the unknot are known to exist. However the following is yet unknown: For any fixed knot type \([K]\), is the growth rate of polygons representing \([K]\) equal to that of the unknot? Such results have been shown by restricting to very specific subclasses of self avoiding polygons, but are not known in the general case.

We do have a weak result for such classes:
\begin{theorem}
  Let \([L]\) and \([K]\) be any two knot types, and \(\ArbClass[[L]]\) the class of either general, prime, or reduced diagrams of knot type \([L]\). Then there exists \(c > 0\) so that the following holds: Let \(\mathscr{H}_n\) be the subclass of diagrams in \(\ArbClass[[L]]_n\) who contain fewer than \( cn \) copies of \([K]\) as a subknot. Then
  \[ \limsup_{n\to\infty}{|\mathscr{H}_n|^{1/n}} <
    \limsup_{n\to\infty}{|\ArbClass[[L]]_n|^{1/n}}. \]
\end{theorem}
\begin{proof}
  The argument here is the same as that for unknot diagrams above: Tangle diagrams which insert slipknots of type \([K]\) can be freely inserted (by methods like connect summation) without changing the knot type of a diagram. We are left with the weak result because we do not know enough about the growth rate of the class of diagrams \(\ArbClass[[L]]\).
\end{proof}

We do conjecture based on numerical evidence and connections to random self-avoiding polygons that:
\begin{conjecture}
  Let \([L]\) be a fixed knot type and \(\ArbClass[[L]]\) be the class of either general, prime, or reduced diagrams of fixed knot type \([L]\), and let \(\arbclass_n[[L]]\) be the count of \(n\)-crossing diagrams in \(\ArbClass[[L]]\). Then the limit
  \[
    \lim_{n\to\infty}{\arbclass_n[[L]]^{1/n}}
  \]
  exists.
\end{conjecture}
If the above were true, then we would get the ``strong'' pattern theorem
result:
\begin{conjecture}
  Let \([L]\) and \([K]\) be any two knot types, and
  \(\ArbClass[[L]]\) the class of either general, prime, or reduced
  diagrams of knot type \([L]\). Then there exists \(c > 0\) and \(0 <
  d < 1\) so that
  \[ \Prb(D \text{ contains \(\le c|D|\) copies of \([K]\) as a
      subknot}) < d^{|D|}, \]
  so that the probability that a random diagram in \(\ArbClass[[L]]\)
  contains the knot type \([K]\) as a subknot goes to one as the
  complexity of the diagram (\textit{i.e.}\ the number of vertices) goes to
  infinity.
\end{conjecture}

Indeed, the belief is that not only do the limits
\[ \lim_{n\to\infty}{\arbclass_n[[L]]^{1/n}} \]
exist, but that for any two knot types \([K], [L]\), the limits are the same, \textit{i.e.},
\[ \lim_{n\to\infty}{\arbclass_n[[K]]^{1/n}} = \lim_{n\to\infty}{\arbclass_n[[L]]^{1/n}}. \]
This is the belief for other models of random models of knotting, namely self avoiding polygons~\cite{Rensburg_1991,Deguchi_1997,Orlandini1998,Rensburg2011}, where there is extensive numerical evidence.

\subsection{In Classes of Other Types}
\label{sec:classes-other-types}

Rather than partitioning diagrams along ``knot invariant'' lines (such as knot type), we can partition along ``diagram invariant'' lines. One such example is the \emph{unknotting number} of a diagram \(D\), which is the minimum number \(\Unk(D)\) of crossings of a fixed diagram that need to be toggled in order to change a given diagram to the unknot. This is related to the usual definition of the unknotting number of a knot type \([K]\) by,
\[ \Unk([K]) = \min_D{\Unk(D)}, \]
where the minimum is over all diagrams \(D\) which represent the knot type \([K]\).

The unknotting number is biologically relevant: Topoisomerase enzymes are tasked with untangling DNA in cells~\cite{Buck04} among other duties. Type-2 topoisomerases work locally on DNA molecules by passing an upper strand of DNA through another; diagrammatically this process is identical to a crossing toggle. In this context, crossing toggles are called \emph{strand passages}, and the unknotting number of a knot type is its distance from \(0_1\) under the \emph{strand passage metric}~\cite{Darcy1997,Soteros2011}.

\begin{theorem}
  Let \(c_n(\ell)\) be the count of \(n\)-crossing (reduced, prime, or general) knot
  diagrams with unknotting number \(\ell\). Then
  \(\lim_{n\to\infty}{(c_n(\ell))^{1/n}} = \tau_0\), the growth constant of
  counts of (prime, reduced, or general) unknot diagrams (\(\tau_0\) depends on the class of diagrams).
\end{theorem}

\begin{proof}
  Let \(\ArbKnotDiaUK_n(\ell)\) denote the set of (reduced or general) rooted knot diagrams with unknotting number precisely \(\ell\). Define subsets \(\ArbKnotDiaUK_n^+(\ell)\) and \(\ArbKnotDiaUK_n^-(\ell)\) to be those diagrams with the property that toggling the root vertex (\textit{i.e.}\ the vertex which the root points to) increases or decreases the unknotting number of the diagram, respectively. Observe from the bijection where one toggles the root vertex that \(\ArbKnotDiaUK_n^+(\ell) \cong \ArbKnotDiaUK_n^-(\ell-1)\).

  Let \(c_n(\ell), c_n^+(\ell),\) and \(c_n^-(\ell)\) be the sizes of the sets \(\ArbKnotDiaUK_n(\ell), \ArbKnotDiaUK_n^+(\ell),\) and \(\ArbKnotDiaUK_n^-(\ell)\) respectively.

  Consider \(\ell > 0\). A toggle of any crossing in a diagram in \(\ArbKnotDiaUK_n(\ell)\) will either increase, decrease, or fix the unknotting number. As \(\ell > 0\), there \emph{must} be at least one crossing whose toggling decreases the unknotting number (this is the definition). Hence \(|c_n(\ell)| \le n|c_n^-(\ell)|\). Then,
  \[\frac{1}{n}|c_n(\ell)| \le |c_n^-(\ell)| \le
    |c_n(\ell)|,\]
  where the latter inequality comes from that one is the subset of the other. This implies then that
  \[ \lim_{n\to\infty}(c_n^-)^{1/n} = \lim_{n\to\infty}(c_n)^{1/n}.\]

  Let \(D\) be a rooted standard diagram representation of the \((2,2\ell+1)\) torus knot. It is known that \(D\) has unknotting number \(\ell\). So there is an injection from \(\ArbKnotDiaUK_n(0) \hookrightarrow \ArbKnotDiaUK_{n+(2\ell+1)}(\ell)\) given by connect-summing \(D\) to each rooted unknot in \(\ArbKnotDiaUK_n(0)\). This implies then that
  \[ \tau_0 = \lim_{n\to\infty}{(c_n(0))^{1/n}} \le
    \lim_{n\to\infty}{(c_{n+(2\ell+1)}(\ell))^{1/n}} =
    \lim_{n\to\infty}{(c_{n}(\ell))^{1/n}} \]

  The theorem then follows by an induction argument. For \(\ell=0\) we have that \(|\ArbKnotDiaUK_n(0)|\) is simply the class of rooted unknot diagrams with \(n\) crossings and the growth rate of this class is \(\tau_0\) by definition.

  On the other hand, suppose that for some \(\ell \ge 0\), \(\lim_{n\to\infty}(c_n(\ell))^{1/n} = \tau_0\). So,
  \[ \tau_0 = \lim_{n\to\infty}{(c_n(\ell))^{1/n}} \ge
    \lim_{n\to\infty}{(c_n^+(\ell))^{1/n}} =
    \lim_{n\to\infty}{(c_n^-(\ell+1))^{1/n}} =
    \lim_{n\to\infty}{(c_n(\ell+1))^{1/n}} \ge \tau_0, \]
  where the final inequality was explained above. Hence \(\tau_0 = \lim_{n\to\infty}{(c_n(\ell+1))^{1/n}}\) as desired.
\end{proof}

Diagrammatic unknotting number satisfies the following relation under connect
summation of two diagrams \(A, B\):
\[ \Unk(A\#B) = \Unk(A) + \Unk(B). \]
This implies that connect summation of an unknotted 2-tangle to a diagram with unknotting number \(\ell\) produces a resultant diagram which is also unknotting number \(\ell\). This means that if \(P\) is an unknotted tangle which is admissible for attachment into diagrams in \(\ArbClass\), it is \emph{also} admissible for attachment into diagrams \(\ArbKnotDiaUK(\ell)\). This then implies;

\begin{theorem}
  \label{thm:unknumslipknots}
  Let \([K]\) be any knot type, and \(\ArbKnotDiaUK(\ell)\) the class of either general or reduced diagrams with unknotting number \(\ell\). Then there exists \(c > 0\) and \(1 > d > 0\) so that
  \[ \Prb(D \text{ contains \(\le c|D|\) copies of \([K]\) as a
      subknot}) < d^{|D|}, \]
  so that the probability that a random diagram in \(\ArbKnotDiaUK(\ell)\) contains the knot type \([K]\) as a subknot goes to one as the complexity of the diagram (\textit{i.e.}\ the number of vertices) goes to infinity.
\end{theorem}
\begin{proof}
  The argument is the same as in the proof of Theorem~\ref{thm:unkslipknots}. Connect sum is a viable attachment scheme as connect summation of an unknot diagram fixes the unknotting number (in fact, the knot-theoretic unknotting number also fixed as the knot type itself is fixed).
\end{proof}

\subsection{Asymptotics}
\label{sec:asymptotics}

We have mentioned prior the conjecture:
\begin{conjecture}
  The asymptotic growth of knotoid diagrams is in fact
\[
  \frac{\subknotdia_n}{2^n} =
  \subknotshad_n \mathop{\sim}\limits_{n \to \infty}
  c'\knotgrowth^n \cdot n^{\gamma' - 1},
\]
for some constant \(\gamma'\) (it is furthermore believed \(\gamma' = \gamma\)).
\end{conjecture}
It is in fact still unknown whether \(\lim_{n\to\infty}{\knotshad_n^{1/n}} = \knotgrowth\), although Theorem~\ref{thm:knotoidgrowth} states that the limit exists. The statement of this conjecture arises from intuition from link shadows and self-avoiding polygons. In this section, we present additional evidence for this conjecture.

\subsubsection{Mean geodesic distance}
\label{sec:geodist}

We start with an argument for which we have numerical support. This involves the geodesic distance of an knotoid diagram:

\begin{definition}
  The \emph{geodesic distance}~\cite{Bouttier2003} (also called \emph{complexity}~\cite{Turaev2012} or \emph{height}~\cite{Gugumcu_2017}) \(d(L)\) of an knotoid diagram \(L\) is the minimum number of edges that any path in the surface of the diagram connecting the two legs crosses.
\end{definition}

This is a concept discussed for more general maps in~\cite{Bouttier2003}. One of their conclusions is that the average geodesic distance between the two open legs in the case of all open link shadows is the same as the average distance between any two faces in a closed link shadow, asymptotically proportional to \(n^{1/4}\). It is believed that the average geodesic distance for knot shadows is the same or smaller, as it is ``more difficult'' for knotoid shadows to have large geodesic distance. For instance, Figure~\ref{fig:max_geodist} compares the smallest knotoid shadow with geodesic distance 3 to the smallest open link shadow with the same geodesic distance.

\begin{figure}
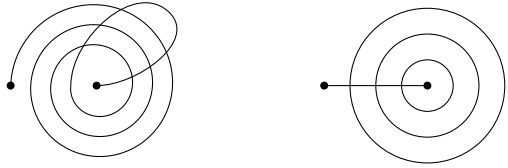

  \centering
  \includestandalone{figs/minimum_geodist}
  \caption{A minimal knotoid shadow with geodesic distance 3 (left) has 6
    crossings. A minimal open link shadow with geodesic distance 3 (right) only
    has 3.}
  \label{fig:max_geodist}
\end{figure}

Preliminary data suggests this conjecture. An alternate closure method for blossom trees of Bouttier, \textit{et al}~\cite{Bouttier_2002} is in bijection with multi-knotoid shadows. Introducing this algorithm into Gilles Schaeffer's \texttt{PlanarMap}~\cite{SchaefferPlanarMap} software we were able to gather data to test this conjecture. Figure~\ref{fig:geodistdata} shows curves for the mean geodesic distance of multi-knotoid diagrams and knotoid diagrams.

\begin{figure}
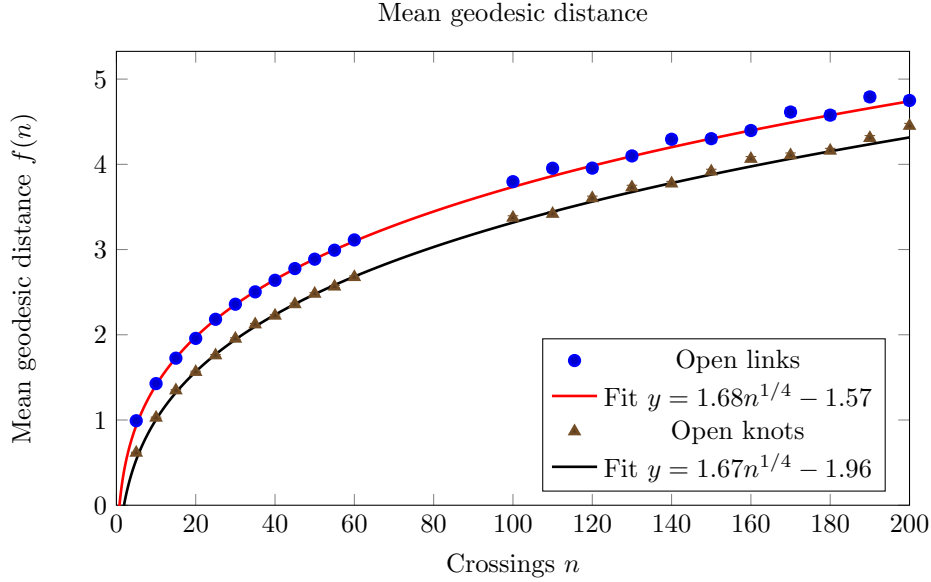

  \centering
  \includestandalone{figs/avg_geodist}
  \caption{The mean geodesic distance of knotoid shadows appears to grow at the same rate as for multi-knotoid diagrams.}
  \label{fig:geodistdata}
\end{figure}

If we assume this, notice that a result on the growth rates follows:
\begin{theorem}
  Let \(f(n)\) denote the average geodesic distance of knotoid shadows in \(\SubKnotShad_n\). Suppose that \(\lim_{n\to\infty}{f(n)/n} = 0\). Then the growth rate of knotoid shadows is,
  \[\lim_{n\to\infty}{\subknotshad_n^{1/n}} = \lim_{n\to\infty}{\knotshad_n^{1/n}} =
    \knotgrowth\]
\end{theorem}
\begin{proof}
  Partition the numbers \(\subknotshad_n\) of knotoid shadows with \(n\) crossings by geodesic distance: Let \(\subknotshad_n[\ell]\) be the number of \(n\)-crossing knotoid shadows of geodesic distance \(\ell\), so that \(\subknotshad_n = \sum_{\ell=0}^n \subknotshad_n[\ell]\) (\(n\) is an upper bound on geodesic distance; \(2n\) is even more trivial as it is the number of edges, and would equally suffice).

  First notice that \(\knotshad_n \le \subknotshad_n\) and hence \(\mu \le \mu_I\) by splitting the root edge of knot shadows counted by \(\knotshad_n\). Notice furthermore that \(\subknotshad_n[\ell] \le \knotshad_{n+\ell}\) as we can deterministically close any distance \(\ell\) knotoid shadow by introducing \(\ell\) new crossings (and this map is inverted on its image by contracting \(\ell\) crossings back from the root).

  Define \(\subknotshad_n[>\ell] = \sum_{j=\lfloor \ell \rfloor+1}^{n}{\subknotshad_n[j]}\) and \(\subknotshad_n[\le \ell] = \sum_{j=0}^{\lfloor \ell \rfloor}{\subknotshad_n[j]}\). For a nonnegative function \(g(n)\) obeying \(\lim_{n\to\infty}{g(n)/n = 0}\), let
  \[ m = \argmax_{\ell \le \lceil g(n) \rceil}\{\subknotshad_n[\ell]\} \] so
  that \( \lceil g(n) \rceil \subknotshad_n[m] \ge \subknotshad_n[\le g(n)]\)
  (observe that trivially \(\lceil g(n) \rceil \ge m\)).

  First, suppose we have some constant \(\alpha\) for which, independently of \(n\), we can partition \(\subknotshad_n =
  \subknotshad_n[\le g(n)] + \subknotshad_n[> g(n)]\) with \(\alpha \subknotshad_n \le \subknotshad_n[\le
  g(n)]\). Say \(g(n)\) is a {\em pivot\/} for
  \(\alpha\). Then \(\subknotshad_n \le \alpha^{-1}\subknotshad_n[\le g(n)]\) and
  furthermore,
  \[ \knotshad_n \le \subknotshad_n \le \alpha^{-1}\subknotshad_n[\le g(n)]
    \le \alpha^{-1}\lceil g(n) \rceil \subknotshad_n[m]
    \le \alpha^{-1}\lceil g(n) \rceil \knotshad_{n+m}
    \le \alpha^{-1} (g(n)+1) \knotshad_{n+\lceil g(n) \rceil} \]
  So,
  \begin{align*}
    \ln(\mu) \le \ln(\mu_I)
    &\le
      \limsup_{n\to\infty}{
      \frac{-\ln(\alpha) + \ln(g(n)+1) + \ln(\knotshad_{n+\lceil g(n) \rceil})}{n}} \\
    &= \limsup_{n\to\infty}{\frac{\ln(\knotshad_{n+\lceil g(n) \rceil})}{n}} \\
    &= \left( \limsup_{n\to\infty}{
      \frac{\ln(\knotshad_{n+\lceil g(n) \rceil})}{n + \lceil g(n) \rceil}} \right)
      \left( \limsup_{n\to\infty}{\frac{n + \lceil g(n) \rceil}{n}} \right) \\
    &\le \ln(\mu),
  \end{align*}
  provided the condition on \(g(n)\) established prior.

  Let \(g(n) \ge 2f(n)\); then \(g(n)\) obeys \(\lim_{n\to\infty}{g(n)/n} = 0\), by the assumption on \(f(n)\). By Markov's inequality~\cite{Flajolet2009}, we have that \(\subknotshad_n[\le g(n)] \ge (f(n)/g(n))\subknotshad_n = \subknotshad_n/2\). Hence \(g(n)\) is a pivot for \(\alpha = 1/2\). This then shows that the growth rate result follows.
\end{proof}
Of course, this is dependent on proving such a bound on the mean geodesic distance of knotoid shadows. No upper bounds appear to be known. On the other hand, the affine index and arrow polynomials of a knotoid diagram provide lower bounds on its geodesic distance~\cite{Gugumcu_2017}, so understanding the average geodesic distance will also lead to new understanding of knotoid invariants.

\subsubsection{Conformal field theory}
\label{sec:cft}

We first reason out the conjecture using a conformal field theory argument similar to that of~\cite{DiFrancesco2000}. Zinn-Justin and Zuber have proposed the following matrix integral~\cite{ZinnJustin2009} for the study of knot diagrams:
\[ Z(N, \tau; x) = \int{\prod_{a=1}^\tau{dM_a}\exp{\left(
        -\frac{N}{2} \Tr{ \left(
            \sum_{a=1}^\tau{M_a^2} -
            \frac{x}{2} \sum_{a,b=1}^\tau{(M_aM_b)^2}
          \right)}
      \right)}}, \]
where the integral is over all \(N \times N\) Hermitian matrices, \(\tau\) is the number of colors with which loops are marked, and \(x\) is a formal variable. What is important is that the formal series
\[ F(x) = \lim_{\tau\to 0}\lim_{N\to\infty}\frac{\log Z(N, \tau; x)}{N^2} \]
is the generating function which counts all knot shadows, \(\sum_n{x^n\knotshad_n}\). Applying the ideas of Di Francesco, \textit{et al.}, one defines the operator
\[ \varphi_1 = \lim_{N\to\infty}\frac{1}{N}\Tr{\sum_{a=1}^\tau{M_a}}, \]
which adds a single node of degree 1. Hence the expectation \(\lim_{\tau\to 0}\frac 1\tau \langle \varphi_1\varphi_1 \rangle\) should be a generating series for the knotoid shadows. One should then, after calculating the conformal dimension of \(\varphi_1\) and using the KPZ formula, be able to conjecture precisely the the critical exponent of \(\gamma'\). We note that the difficulty in this case may be greater due to the inverse logarithmic correction term in the conjectured asymptotic growth of knot diagrams~\cite{PZJPrivateComm}.

This argument is non-rigorous, and any results depend on the validity of the mathematical physical interpretation.

\subsubsection{Enumeration of virtual knot diagrams}
\label{sec:virtknots}

Here is another argument for the exponential growth term
\[ \lim_{n\to\infty} \subknotshad_n^{1/n} = \knotgrowth \]
in the conjecture. Consider a knotoid shadow. Its two loose ends either lie in the same face, or in two different faces. This leads to the decomposition of the class of knotoid shadows into
\[ \SubKnotShad = \SubKnotShad^0 \cup \SubKnotShad^1, \]
where \(\SubKnotShad^0\) consists of all knotoid diagrams whose loose ends lie in the same face and \(\SubKnotShad^1\) consists of all knotoid diagrams whose loose ends lie in different faces. Observe that \(\SubKnotShad^0\) is equivalent to \(\KnotShad\) by joining the loose ends to form the root edge (whose direction is determined by the order of the loose ends which were joined).

On the other hand, consider a knotoid shadow \(D\) in \(\SubKnotShad^1\); by adding a handle connecting both of the faces in \(D\) with loose ends, one is able to produce a new rooted knot shadow \(D'\) on the \emph{torus} by now joining the two loose ends in \(D\) by way of the added handle\footnote{Knot diagrams on arbitrary surfaces are called \emph{virtual knot diagrams}~\cite{Kauffman1999} and are themselves worthy of study.}; observe that all faces in \(D'\) are still disks as critical for the definition of a map. This produces an injection from \(\SubKnotShad^1\) into \(\KnotShad^T\), the class of knot shadows on the torus.

For many different classes of maps \(\ArbClass^S\) on a fixed surface
\(S\), it has been shown~\cite{Bender1986,Gao_1991,Gao_1993} that
\[ \arbclass^S_n \sim \kappa(\ArbClass,S) \tau({\ArbClass})^n n^{\gamma(\ArbClass,\chi(S))}, \]
where \(\kappa(\ArbClass,S)\) is a constant depending on the class and the surface, \(\tau(\ArbClass)\) is a constant depending on the class, and \(\gamma(\ArbClass,\chi(S))\) is a constant depending on the class and the Euler characteristic of the surface. This leads one to conjecture,
\begin{conjecture}
  The asymptotic growth of knot diagrams on an arbitrary fixed surface
  \(S\) is conjectured to be
\[
  \frac{\knotdia^S_n}{2^n} =
  \knotshad^S_n \mathop{\sim}\limits_{n \to \infty}
  c''\knotgrowth^n \cdot n^{\gamma''(\KnotShad,\chi(S))},
\]
where \(\gamma''(\KnotShad,\chi(S))\) is a constant depending only on the Euler
characteristic of the surface \(S\) (it is believed that
\(\gamma''(\KnotShad,\chi(S)) = ((\gamma-2)(\chi(S)/2))\)).
\end{conjecture}
If in fact this conjecture is true, that \(\SubKnotShad\) decomposes into rooted knot shadows and a subset of rooted knot shadows on the torus would imply that the exponential growth terms for knotoid shadows is the same as knot shadows. We note though that the original proofs of~\cite{Bender1986} rely on \emph{first} proving results on enumerations of maps analogous in their settings to knotoid shadows (so this method of reasoning is likely backwards).

Finally, we mention a connection of this problem to the study of oriented Gauss diagrams (also known as oriented chord diagrams). Oriented Gauss diagrams are in bijection with combinatorial curves on arbitrary orientable surfaces~\cite{Carter_1991}. Oriented Gauss diagrams have a different class of growth (there are \((2n)!/n!\) with \(n\) chords~\cite{Linial10}), so this connection alone is insufficient for proving the conjectures of this section. Having control over the genus of such objects could lead to understanding whether the exponential growth rate \(\mu_K\) depends on \(g\) and possibly even relations for the power law term \(\gamma''(\KnotShad, g)\). It is known that the genus of a random oriented Gauss diagram is large~\cite{Linial10}: \(\Theta(n - \log n)\). We note that in the case of \emph{unoriented} Gauss diagrams with a \emph{different definition} of genus, Harer and Zagier~\cite{Zagier1986} present a generating function which enumerates unoriented Gauss diagrams by genus. It is not known how these two definitions are related.

\section{Conclusion}
\label{sec:conclusion}

% If we consider diagrams which are not necessarily smooth, i.e.\ they
% perhaps have additionally vertices of degree 2 (called ``kinks''), we
% can obtain a new view of subknotting which is only numerically
% different.

% \begin{definition}
%   Let \(S\) be a (possibly not smooth) knotoid diagram (of at least one
%   vertex) and \(a, z\) its loose flags. We wish to define the
%   \emph{contraction \(\FlagContract{S}{a}\) of \(S\) by flag \(a\)}.

%   If \(S\) consists of exactly two flags and one kink, then its
%   contraction is the trivial knotoid diagram.

%   If that flag \(a\) is connected to a crossing \(v\) joining flags
%   \(a,b_1,c_1,d_1\), oriented counterclockwise (namely, flag \(c_1\)
%   meets the crossing opposite \(a\)). Suppose that the latter three
%   flags are part of edges \(e_b = (b_1b_2)\), \(e_c = (c_1c_2)\), and
%   \(e_d = (d_1d_2)\). Then the contraction is the new (certainly not
%   smooth) knotoid diagram produced by deleting crossing \(v\), edge
%   \(e_c\), and flags \(a\) and \(c_1\) and inserting the
%   kink \(w\) which joins flags \(b_1\) and \(d_1\).

%   If instead \(a\) is connected to a kink \(v\) joining flags
%   \(a, c_1\) and \(c_1\) is part of the edge \(e = (c_1c_2)\), then
%   the contraction is the new knotoid diagram produced by deleting the
%   vertex \(v\) and flags \(a_2\) and \(c_1\).

%   Regardless, \(\FlagContract{S}{a}\) is a new knotoid diagram with one
%   fewer edge and loose flags \(c_2\) and \(z\).
% \end{definition}

We can consider slightly different definitions for contraction while preserving results in this paper. Under one such alternate definition that produces ``non-smooth'' diagrams with some 2-valent vertices, we can produce graphics (Figures~\ref{fig:subspectrum} and~\ref{fig:subspectrum51}) describing the appearances of subknots and slipknots inside of knot diagrams similar to those of~\cite{Rawdon2015,Rawdon2012}. In these \emph{disk matrices}, the (normalized) radius represents the fraction of the diagram used as the subdiagram; the other circle hence describes the knot type of the whole diagram. The angle corresponds to the edge index, ordered around the knot diagram with an orientation, from which the diagram is split and contracted.

We plan to further investigate the relation between results for space curves and those for knot diagrams. It is encouraging how similar these figures are to Rawdon's. Namely, observe how slipknots represent as tendrils starting near the outer rim and slowly winding towards the center before vanishing.

\begin{figure}[hbtp]
  \centering
  \begin{subfigure}[hbtp]{.4\linewidth}
    \centering
    \scalebox{.5}{
      \input{figs/subknot_dia3.tikz}
    }
  \caption{A knot diagram of knot type \(4_1\#3_1\).}
  \label{fig:diagramsub}
  \end{subfigure}\hfil
  \begin{subfigure}[hbtp]{.4\linewidth}
    \centering
    \includegraphics[height=2.3in,viewport=120 40 480 400,clip=true]{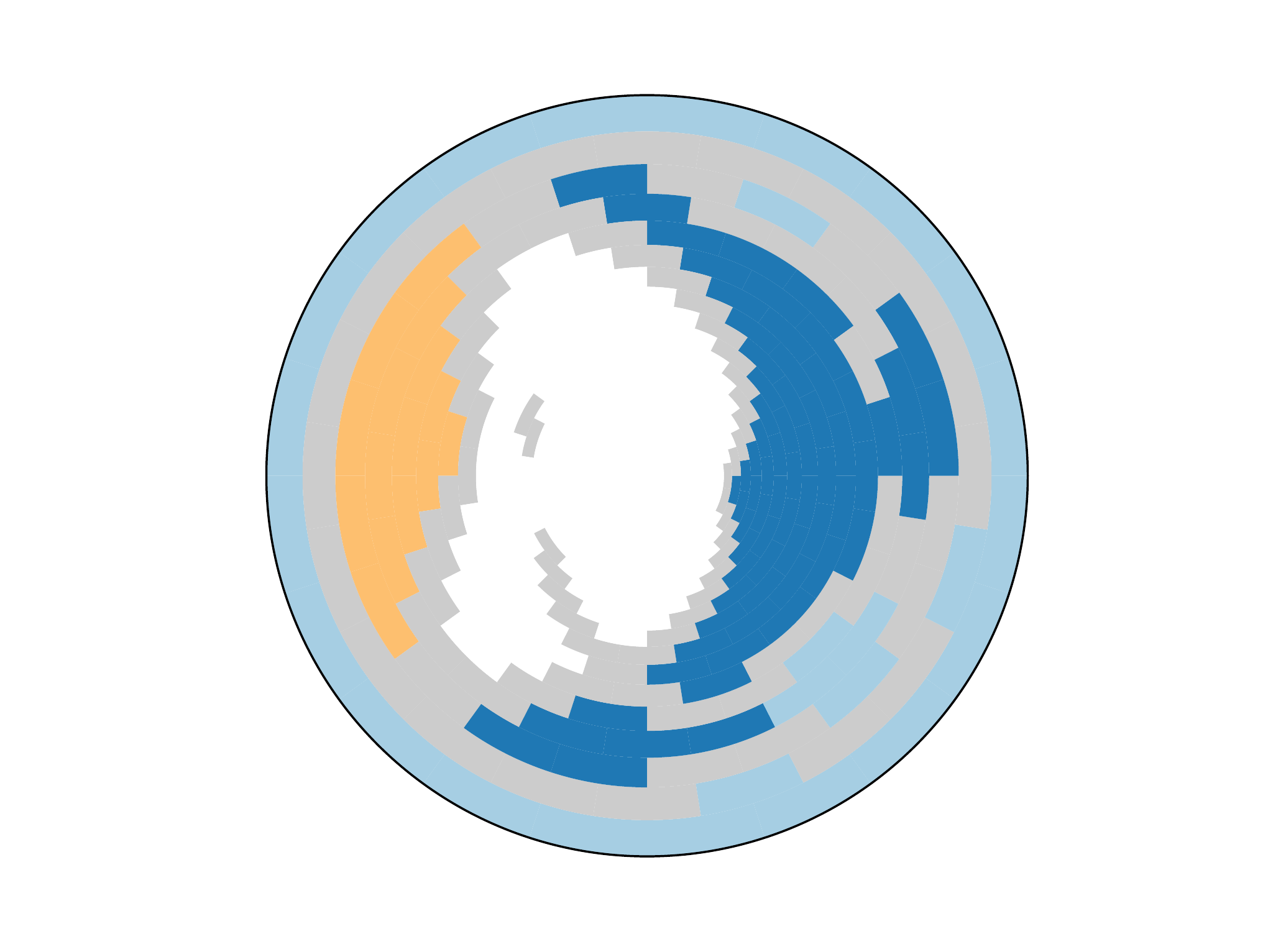}
    \caption{Disk matrix of subknots in the diagram
      in~\ref{fig:diagramsub}. Light blue is \(4_1\#3_1\), orange is
      \(4_1\), and dark blue is \(3_1\). Mixed knot types are gray
      and unknots are white.}
  \label{fig:subspectrum}
  \end{subfigure}
\end{figure}

\begin{figure}[hbtp]
  \centering
  \begin{subfigure}[hbtp]{.4\linewidth}
    \centering
    \scalebox{.5}{
      \input{figs/subknot_5_1_dia.tikz}
    }
  \caption{The standard knot diagram for the torus knot \(5_1\).}
  \label{fig:diagramsub51}
  \end{subfigure}\hfil
  \begin{subfigure}[hbtp]{.4\linewidth}
    \centering
    \includegraphics[height=2.3in,viewport=120 40 480 400,clip=true]{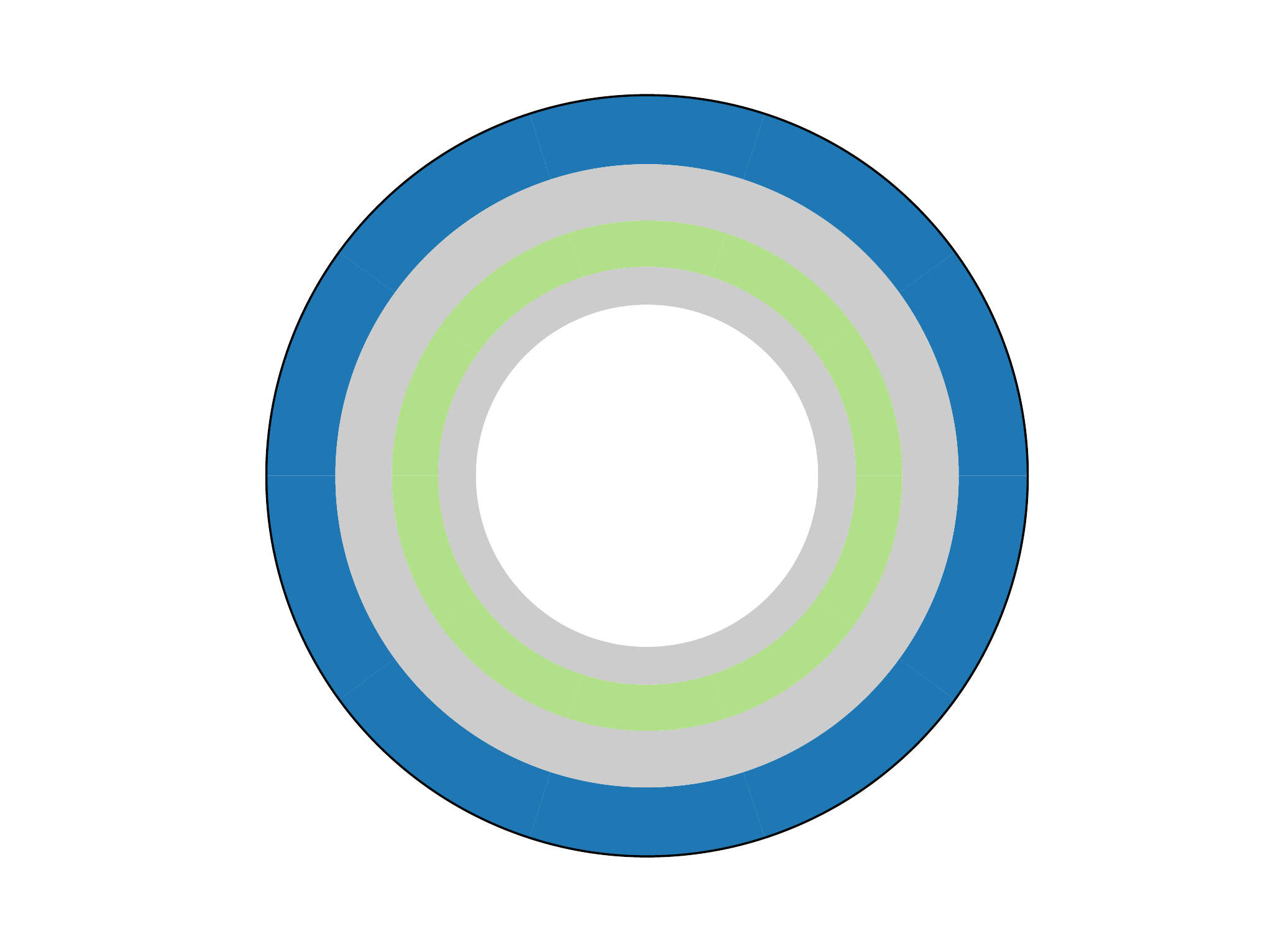}
    \caption{Disk matrix of subknots in the diagram
      in~\ref{fig:diagramsub51}. Blue is \(5_1\), orange is
      \(4_1\), and light green is \(3_1\). Mixed knot types are
      gray and unknots are white. Cf.\ figure 2b
      in~\cite{Rawdon2015}.}
  \label{fig:subspectrum51}
  \end{subfigure}
\end{figure}

Ideally we hope that results which we can show in the random diagram
model---such as that almost all unknot diagrams are slipknotted---can
be transferred to other models of knots with ease. It is promising
though that other results (and the methods used in proving them) for
different models of random knots apply for random knot
diagrams. We hope that for some other ``universal'' properties we can
move in the other direction.

Indeed it is still an open question of which type of random knotting,
e.g.\ whether the knots are constrained to a sphere, slab, or tube,
the random knot diagram model best approximates. The same is true for
other classes of knot diagrams: We have discussed reduced diagrams and
prime diagrams but expect that results hold for further classes such
as 6-edge-connected diagrams that have no immediate ``knot theoretic''
connection.

\FloatBarrier

\begingroup
\raggedright{}
\sloppy
\printbibliography{}
\endgroup

\end{document}

%%% Local Variables:
%%% mode: latex
%%% TeX-master: t
%%% TeX-command-extra-options: "-shell-escape"
%%% End:

%% file: figs/subknot_dia3.tikz
\definecolor{linkcolor0}{rgb}{0.85, 0.15, 0.15}
\begin{tikzpicture}[line width=2.4, line cap=round, line join=round]
  \begin{scope}[color=black
  ]
    \draw (4.42, 5.33) -- (2.52, 5.33) -- (2.52, 4.29);
    \draw (2.52, 4.29) -- (2.52, 3.44);
    \draw (2.52, 3.06) -- (2.52, 2.40);
    \draw (2.52, 2.01) -- (2.52, 1.16) -- (4.61, 1.16) -- (4.61, 4.10);
    \draw (4.61, 4.48) -- (4.61, 5.33);
    \draw (4.61, 5.33) -- (4.61, 6.38) -- (9.83, 6.38) -- (9.83, 2.21) -- (7.74, 2.21) -- (7.74, 3.25);
    \draw (7.74, 3.25) -- (7.74, 4.10);
    \draw (7.74, 4.48) -- (7.74, 5.33) -- (6.70, 5.33) -- (6.70, 4.29);
    \draw (6.70, 4.29) -- (6.70, 3.25) -- (7.55, 3.25);
    \draw (7.93, 3.25) -- (8.78, 3.25) -- (8.78, 4.29) -- (7.74, 4.29);
    \draw (7.74, 4.29) -- (6.89, 4.29);
    \draw (6.50, 4.29) -- (5.84, 4.29);
    \draw (5.46, 4.29) -- (4.61, 4.29);
    \draw (4.61, 4.29) -- (2.71, 4.29);
    \draw (2.33, 4.29) -- (1.48, 4.29) -- (1.48, 3.25);
    \draw (1.48, 3.25) -- (1.48, 2.21) -- (2.52, 2.21);
    \draw (2.52, 2.21) -- (3.57, 2.21) -- (3.57, 3.25) -- (2.52, 3.25);
    \draw (2.52, 3.25) -- (1.67, 3.25);
    \draw (1.29, 3.25) -- (0.43, 3.25) -- (0.43, 0.12) -- (5.65, 0.12) -- (5.65, 4.29);
    \draw (5.65, 4.29) -- (5.65, 5.33) -- (4.80, 5.33);
  \end{scope}
\end{tikzpicture}

%% file: figs/subknot_5_1_dia.tikz
\definecolor{linkcolor0}{rgb}{0.85, 0.15, 0.15}
\begin{tikzpicture}[line width=2.4, line cap=round, line join=round]
  \begin{scope}[color=black]
    \draw (7.48, 7.16) -- (7.48, 9.50) -- (2.79, 9.50) -- (2.79, 7.35);
    \draw (2.79, 6.97) -- (2.79, 4.81);
    \draw (2.79, 4.81) -- (2.79, 2.47) -- (4.94, 2.47);
    \draw (5.33, 2.47) -- (7.48, 2.47);
    \draw (7.48, 2.47) -- (9.82, 2.47) -- (9.82, 7.16) -- (7.67, 7.16);
    \draw (7.29, 7.16) -- (2.79, 7.16);
    \draw (2.79, 7.16) -- (0.45, 7.16) -- (0.45, 4.81) -- (2.60, 4.81);
    \draw (2.98, 4.81) -- (5.13, 4.81) -- (5.13, 2.47);
    \draw (5.13, 2.47) -- (5.13, 0.13) -- (7.48, 0.13) -- (7.48, 2.28);
    \draw (7.48, 2.66) -- (7.48, 7.16);
  \end{scope}
\end{tikzpicture}